\documentclass[11pt, a4paper]{amsart}
\usepackage{amssymb}
\usepackage{amsmath}
\usepackage{amssymb}
\usepackage{amsthm}
\usepackage{bbm}
\usepackage{bm}
\usepackage{mathtools}
\usepackage{mathrsfs}
\usepackage[T1]{fontenc}
\usepackage[usenames,dvipsnames,svgnames,table]{xcolor}
\usepackage{esint}
\allowdisplaybreaks
\usepackage{combelow}
\usepackage{enumitem}
\usepackage[colorlinks=true,citecolor=black,urlcolor=NavyBlue,
linkcolor=black]{hyperref}

\usepackage{graphicx}
\usepackage{tikz}

\usepackage{geometry}
\makeatletter
\newcommand*\bigcdot{\mathpalette\bigcdot@{.5}}
\newcommand*\bigcdot@[2]{\mathbin{\vcenter{\hbox{\scalebox{#2}{$\m@th#1\bullet$}}}}}
\makeatother

\newtheorem{theorem}{Theorem}

\newtheorem{corollary}[theorem]{Corollary}

\theoremstyle{remark}

\usepackage{booktabs}

\normalsize

\def\XXint#1#2#3{{\setbox0=\hbox{$#1{#2#3}{\int}$ }
		\vcenter{\hbox{$#2#3$ }}\kern-.6\wd0}}

\newcommand{\dif}{\operatorname{d}\!}

\newcommand{\R}{\mathbb{R}}

\newcommand{\lo}{{\operatorname{loc}}}

\renewcommand{\geq}{\geqslant}

\newcommand{\curl}{\operatorname{curl}}
\renewcommand{\leq}{\leqslant}

\newcommand{\lin}{\operatorname{Lin}}

\newcommand{\dist}{\operatorname{dist}}

\usepackage{tikz-cd}
\author{Bogdan Rai\cb{t}\u{a}}

\begin{document}
\title[Quasiconvexity and self-improving size estimates]{Quasiconvexity and self-improving size estimates}
\begin{abstract}
    We show that M\"uller's $L\log L$ bound 
    $$
    F(Du)\geq 0,\,Du\in L^p_{\mathrm{loc}}(\mathbb{R}^n)\implies F(Du)\in L\log L_{\mathrm{loc}}(\mathbb{R}^n)
    $$
    for $F =\det$ and $p=n$ holds for quasiconcave $F$ which are homogeneous of degree $p>1$. This contrasts similar Hardy space bounds which hold only for null Lagrangians.
\end{abstract}
\maketitle
Let $\Omega\subset\R^n$ be an open set. Recall that a homeomorphism $u\in W^{1,n}_{\lo}(\Omega,\R^n)$ is \textit{quasiconformal} if there is a constant $K\geq 1$ such that
\begin{align}\label{eq:qr}
    |Du|^n\leq K\det Du\quad\text{a.e. in }\Omega,
\end{align}
where $|\,\bigcdot\,|$ denotes the operator norm of an  matrix. A far reaching result of Gehring \cite{gehring1973p,bojarski1955homeomorphic}, states that there exists  $\varepsilon_0>0$ such that $u\in W^{1,n+\varepsilon}_\lo(\Omega,\R^n)$ for all $\varepsilon\in(0,\varepsilon_0)$ and, in fact, the homeomorphism assumption is not necessary \cite{meyers1975some}.
A remarkable contribution of M\"uller \cite{muller1989surprising,Muller_90} provides a stunningly simple explanation for this improvement by proving that if the Jacobian of $u$ is non-negative, it enjoys the surprising higher integrability property $\det Du\in L\log L_\lo(\Omega)$.

Here we show that only the quasiconcavity property that the determinant enjoys implies the higher integrability, so that
the class of nonlinear functions for which this type of surprisingly improved integrability holds is substantially broader:
\begin{theorem}\label{thm:main}
    Let $F\colon\R^{m\times n}\to\R$ be quasiconcave and $p$-homogeneous for $p>1$. 
 Then 
  $$  
  \begin{rcases}
        Du\in L^p_\lo(\Omega,\R^{m\times n})\\
        F(Du)\geq 0\text{ a.e.}
    \end{rcases}\implies F(Du)\in L\log L_\lo(\Omega),
    $$
    with the estimate 
    $$
    \|F(Du)\|_{L\log L(\omega)}\leq c(\|Du\|_{L^p(\tilde\omega)},\omega,\tilde\omega),
    $$
    where $\omega\Subset\tilde\omega\Subset\Omega$ are open sets. Moreover, there exists $\varepsilon_0>0$ such that for all $\varepsilon\in(0,\varepsilon_0)$
    $$  
  \begin{rcases}
        Du\in L^p_\lo(\Omega,\R^{m\times n})\\
        F(Du)\geq c |Du|^p\text{ a.e.}
    \end{rcases}\implies Du\in L^{p+\varepsilon}_\lo(\Omega)
    $$
    with an analogous estimate.
\end{theorem}
We recall that $F$ is \textit{quasiconcave} \cite{morrey1952quasi} if, for a cube $Q\subset\R^n$,
\begin{align}\label{eq:qc}
    \fint_{Q} F(D\varphi)\dif x\leq F\left(\fint_{Q} D\varphi\dif x\right)\quad\text{for }\varphi\in C^\infty(\R^n,\R^m)\text{ with $Q$-periodic $D\varphi$}.
\end{align}
This inequality is a {one-sided} condition compared with the case of the Jacobian, for which a {two-sided} condition holds. Indeed, if $F=\det$, then \eqref{eq:qc} holds with equality as the determinant is a null Lagrangian.

Theorem~\ref{thm:main} is  relevant to understanding concentration effects of weakly convergent sequences in Sobolev spaces:
\begin{corollary}
     Let $F\colon\R^{m\times n}\to\R$ be quasiconcave and $p$-homogeneous for $p>1$. If a  sequence of gradients $Du_j$  is bounded in $L^p_\lo(\Omega,\R^{m\times n})$  and satisfies $F(Du_j)\geq 0$ a.e. in $\Omega$, then $F(Du_j)$ is locally uniformly integrable in $\Omega$. If in addition $F(Du_j)\geq c|Du_j|^p$ a.e. in $\Omega$, then $|Du_j|^p$ is locally uniformly integrable in $\Omega$.
\end{corollary}
The result of Theorem~\ref{thm:main} is striking when compared to similar Hardy space bounds which hold for signed quantities. It was shown in \cite{CLMS} that if $u\in W^{1,n}_\lo(\Omega,\R^n)$, then $\det Du\in \mathscr H^1_\lo(\Omega)$; moreover, the Jacobian is essentially the only nonlinearity with this property. 

The equivalence between quasiconcavity/quasiconvexity and weak convergence in Sobolev spaces is well understood \cite{kinderlehrer1994gradient,kristensen1999lower,FMP,kristensen2010characterization}, as is the connection between strong quasiconvexity and higher integrability of minimizers by Caccioppolli inequalities \cite{evans1986quasiconvexity}. Through Theorem~\ref{thm:main}, we connect quasiconvexity with quantifying the uniform integrability of nonlinear functionals applied to weakly convergent sequences. 

The purpose of this note is to present the short proof of Theorem~\ref{thm:main}  and to announce several generalizations and consequences that will be proved in \cite{Raita2025}. 

The first case of interest is the extension to nonlinearities acting on vector fields constrained by underdetermined linear pde. Let $A$ be a  linear pde operator with constant coefficients that is homogeneous of degree $k$,
\begin{align*}
    A(\partial)=\sum_{|\alpha|=k}A_\alpha\partial^\alpha,
\end{align*}
where $A_\alpha\in\lin(V,W)$ for finite dimensional normed spaces $V,\,W$.  
  We assume that $A$ has constant rank and spanning wave cone,
  \begin{align*}
      \mathrm{rank\,} A(\xi)=c\quad\text{for }\xi\in\R^n\setminus\{0\}\qquad\text{and}\qquad\,\mathrm{span\,}\{\ker A(\xi)\}_{\xi\in\R^n\setminus\{0\}}=V,
  \end{align*}
  which are standard assumptions in the investigation of weak convergence of $A$-free fields; the former is a meaningful technical limitation, while the latter is mild, see for instance \cite{kristensen2022oscillation}. An integrand $F\colon V\to \R$ is said to be \textit{$A$-quasiconcave} if  inequality \eqref{eq:qc} holds with $D\varphi$ replaced by $Q$-periodic fields $\psi\in C^\infty(\R^n,V)$ with $A\psi=0$.
\begin{theorem}
Let $F\colon V\to \R$ be $A$-quasiconcave and $p$-homogeneous for $p>1$. Let $q\geq 1$ be such that $q>{np}/(n+kp)$.
 Then 
  $$  
  \begin{rcases}
        v\in L^p_\lo(\Omega,V)\\
        A v\in L^q_{\lo}(\Omega,W)\\
        F(v)\geq 0\text{ a.e.}
    \end{rcases}\implies F(v)\in L\log L_\lo(\Omega).
    $$
    \end{theorem}
    This shows the generality of our observation. Much as the class of weakly convergent $A$-free sequences can be identified with the class of all $A$-quasiconvex functions \cite{FM99,kristensen2020introduction,kristensen2022oscillation}, so is $A$-quasiconcavity relevant to the self-improvement $F(v)\in L^1\implies F(v)\in L\log L$ if $Av=0$.

    The idea of proof is flexible enough so that the $L\log L$ improvement can also be obtained in examples where quasiconcavity is not known. A map $u\in W^{1,n}_\lo(\Omega,\R^n)$ that satisfies \eqref{eq:qr} for some $K\geq1$ is said to be \textit{$K$-quasiregular}. Let $p=2K/(K-1)$ and $n=2$, so $\Omega\subset\R^2$. The \textit{Burkholder functional} $B(A)=(K\det A-|A|^2)|A|^{p-2}$ was seen to be intimately connected with martingale inequalities, geometric function theory, and Morrey's problem.  An impressive partial result towards establishing the quasiconcavity of $B$ was recently established in \cite{AFGKK1}, where an extensive description of the relevance of the Burkholder functional is provided.
    \begin{theorem}\label{thm:Burkholder}
            Let $B\colon\R^{2\times 2}\to\R$ be the  integrand defined above with $p=\tfrac{2K}{K-1}$, $K>1$. 
 Then 
  $$  
  \begin{rcases}
        Du\in L^p_\lo(\Omega,\R^{2\times 2})\\
        B(Du)\geq 0\text{ a.e.}
    \end{rcases}\implies B(Du)\in L\log L_\lo(\Omega),
    $$
    Moreover, if $u_j\rightharpoonup u$ in $W^{1,p}_\lo(\Omega,\R^2)$ is a sequence of $K$-quasiregular maps, then $B(Du_j)$ is locally uniformly integrable.
    \end{theorem}
    This result clarifies a recent assertion from \cite[p.1616]{AFGKK2}. Since it is not known whether the Burkholder functional is quasiconcave on $\R^{2\times 2}$, Theorem~\ref{thm:Burkholder} does not follow from Theorem~\ref{thm:main}.
\begin{proof}[Proof of Theorem~\ref{thm:main}]
Let $\omega\Subset\tilde\omega\Subset\Omega$ be open sets
and  consider  $Q_{R/2}\coloneqq Q_{R/2}(x_0)\subset\omega$, the cube centered at $x_0$, be such that $R<2\dist(\partial\omega,\partial\tilde \omega)$. Moreover, consider cut-off functions $\rho_R\in C_c^\infty(Q_R,[0,1])$ such that $|D\rho_R|\leq c/R$, where $Q_R\coloneqq Q_R(x_0)$. Also consider exponents $1<\alpha<p<\beta$ such that $\alpha(\beta-1)\geq \beta(p-1)$ and $\beta^{-1}\geq\alpha^{-1}-n^{-1}$.
The latter inequality ensures that $L^{\alpha}(Q_R)\hookrightarrow W^{-1,\beta}(Q_R)$.

We perform the Helmholtz decomposition into periodic matrix fields
\begin{align*}
    \rho_R Du=DU_R+\eta_R,
\end{align*}
where the solenoidal field $\eta_R$ satisfies $\int_{Q_R}\eta_R\dif x=0$ and we have the estimates
\begin{align}\label{eq:est}
    \|DU_R\|_{L^\alpha(Q_R)}\leq c\|Du\|_{L^\alpha(Q_R)}\quad\text{and}\quad\|\eta_R\|_{L^\beta(Q_R)}\leq cR^{n/\beta-n/\alpha}\|Du\|_{L^\alpha(Q_R)}.
\end{align}
This decomposition follows, for instance, from \cite{FM99} with $\mathcal A=\curl$ and the Sobolev embedding applied to the second inequality. 

Using the homogeneity of $F$ and the positivity of $F(Du)$ and $\rho_R$, we can  estimate
    \begin{align*}
        \fint_{Q_{R/2}}F(Du)\dif x&\leq c\fint_{Q_{R}}F(\rho_RDu)\dif x\leq c\fint_{Q_R}F(DU_R)\dif x+c\fint_{Q_R}|F(\rho_RDu)-F(DU_R)|\dif x.
    \end{align*}    
    We next use the quasiconcavity of $F$ on the first term  and the local Lipschitz bound for $p$-homogeneous quasiconvex integrands 
    $$
    |F(z_1)-F(z_2)|\leq c|z_1-z_2|(|z_1|+|z_2|)^{p-1}\quad\text{for }z_1,\,z_2\in\R^{m\times n}
    $$
    to obtain
    \begin{align*}
        \fint_{Q_{R/2}}F(Du)\dif x\leq cF\left(\fint_{Q_R}DU_R\dif x\right)+c\fint_{Q_R}|\eta_R|(|\rho_RDu|+|DU_R|)^{p-1}\dif x.
    \end{align*}
    We  apply the homogeneity and continuity of $F$ together with the integral identity for $\eta_R$ to the first term and H\"older's inequality for $\beta$ and its conjugate $\beta'=\beta/(\beta-1)$ to the second to get
    \begin{align*}
          \fint_{Q_{R/2}}F(Du)\dif x&\leq c\left|\fint_{Q_R} \rho_RDu\dif x\right|^p\\&+c\left(\fint_{Q_R}|\eta_R|^\beta\dif x\right)^{1/\beta}\left(\fint_{Q_R}(|\rho_RDu|+|DU_R|)^{\beta'(p-1)}\dif x\right)^{1/\beta'}.
    \end{align*}
    We can apply the triangle inequality to the first term and Young's inequality with exponents $p$ and conjugate $p'$ to the second to obtain
    \begin{align*}
          \fint_{Q_{R/2}}F(Du)\dif x&\leq c\left(\fint_{Q_R} |Du|\dif x\right)^p+c\left(\fint_{Q_R}|\eta_R|^\beta\dif x\right)^{p/\beta}\\&+c\left(\fint_{Q_R}(|\rho_RDu|+|DU_R|)^{\beta'(p-1)}\dif x\right)^{p'/\beta'}.
    \end{align*}
    We control the second term using the second estimate in \eqref{eq:est} and use H\"older's inequality for the second since $\alpha\geq\beta'(p-1)$ to get
    \begin{align*}
          \fint_{Q_{R/2}}F(Du)\dif x&\leq c\left(\fint_{Q_R} |Du|\dif x\right)^p+c\left(\fint_{Q_R}|Du|^\alpha\dif x\right)^{p/\alpha}\\&+c\left(\fint_{Q_R}(|\rho_RDu|+|DU_R|)^{\alpha}\dif x\right)^{p/\alpha}.
    \end{align*}
    We next use the H\"older inequality again to absorb the first term in the second and split the third term as follows
    \begin{align*}
          \fint_{Q_{R/2}}F(Du)\dif x&\leq c\left(\fint_{Q_R}|Du|^\alpha\dif x\right)^{p/\alpha}+c\left(\fint_{Q_R}|\rho_RDu|^{\alpha}\dif x\right)^{p/\alpha}+c\left(\fint_{Q_R}|DU_R|^{\alpha}\dif x\right)^{p/\alpha}.
    \end{align*}
    Now both the second and third terms can be absorbed in the first, by using boundedness of $\rho_R$ and the first estimate in \eqref{eq:est}, leading to
    \begin{align}\label{eq:MF}
          \fint_{Q_{R/2}}F(Du)\dif x&\leq c\left(\fint_{Q_R}|Du|^\alpha\dif x\right)^{p/\alpha}
    \end{align}
  for all cubes $Q_{R/2}=Q_{R/2}(x_0)\subset\omega$ with $R<2\dist(\partial\omega,\partial\tilde\omega)$, giving the the crucial estimate for  our proof. From here we can conclude as in \cite{muller1989surprising}. If $R$ is larger, we have the trivial estimate 
  \begin{align}\label{eq:bigR}
  \fint_{Q_{R/2}}F(Du)\dif x\leq c(\omega,\tilde\omega) \|Du\|_{L^p(\omega)}^p\quad\text{for }R\geq 2\dist(\partial\omega,\partial\tilde\omega).
  \end{align}
   Taking supremum over $R<2\dist(\partial\omega,\partial\tilde\omega)$ in \eqref{eq:MF} and writing $Mf$ for the maximal operator of a locally integrable function $f$, we obtain that 
    \begin{align*}
        \sup_{0<R<2\dist(\partial\omega,\partial\tilde\omega)}\fint_{Q_{R/2}}F(Du)\dif x\leq c\left((M[|Du|^{\alpha}\chi_{\tilde\omega}])^{p/\alpha}\right),
    \end{align*}
    which together with \eqref{eq:bigR} and the positivity of $F(Du)$ gives
    \begin{align*}
        M[F(Du)\chi_\omega](x_0)\leq c(\omega,\tilde\omega)\left((M[|Du|^{\alpha}\chi_{\tilde\omega}](x_0))^{p/\alpha}+\|Du\|_{L^p(\omega)}^p\right).
    \end{align*}
    We integrate this to obtain
    \begin{align*}
       \int_{\tilde\omega}M[F(Du)\chi_\omega]\dif x_0\leq c(\omega,\tilde\omega)\left(\int_{\tilde\omega}(M[|Du|^{\alpha}\chi_{\tilde\omega}])^{p/\alpha}\dif x_0+\|Du\|_{L^p(\omega)}^p\right).
    \end{align*}
     By boundedness of the maximal function on $L^p(\R^n)$, we infer
    \begin{align*}
        \int_{\tilde\omega}M[F(Du)\chi_\omega]\dif x_0\leq c(\omega,\tilde\omega)\int_{\tilde\omega}|Du|^{p}\dif x_0,
    \end{align*}
   from which the proof can be concluded using $F(Du)\geq0$ and \cite[Lem.~3]{muller1989surprising}.
\end{proof}

\bibliographystyle{amsplain}
\bibliography{ref}
\end{document}